\documentclass[12pt]{amsart}
\usepackage{graphicx}
\usepackage{pslatex}
\usepackage{amsmath,amsfonts}
\usepackage{rotating}
\usepackage{amssymb}
\usepackage{verbatim}
\usepackage{rotating}
\usepackage{color}
\usepackage{mathrsfs}
\usepackage{hyperref}

\newtheorem{theorem}{Theorem}[section]
\newtheorem*{oldtheorem}{Theorem}

\theoremstyle{definition}

\theoremstyle{remark}

\def\R{\mathbb R}
\def\C{\mathbb C}

\def\N{\mathbb N}
\def\S{\mathscr S}

\def\({\left(}
\def\){\right)}
\def\[{\left[}
\def\]{\right]}
\def\<{\left<}
\def\>{\right>}
\def\less{\lesssim}

\begin{document}

\title{A Note on a Multilinear Local $Tb$ Theorem for Calder\'on-Zygmund Operators}

\author{Jarod Hart}
\address{Higuchi Biosciences Center \\ University of Kansas \\ Lawrence, KS 
}
\email{jvhart@ku.edu}
\author{Lucas Oliveira}
\address{Department of Mathematics, Universidade Federal do Rio Grande do Sul \\ Porto Alegre \\ Rio Grande do Sul, Brazil }
\email{lucas.oliveira@ufrgs.br}

\thanks{}



\date{\today}

\dedicatory{ }




\maketitle

\begin{abstract}
In this short note, we extend a local $Tb$ theorem that was proved in \cite{GHO} to a full multilinear local $Tb$ theorem.
\end{abstract}

\maketitle

\section{Introduction}

In \cite{GHO}, we proved a multilinear local $Tb$ theorem for square functions, and applied it to prove a local $Tb$ theorem for singular integrals.  The local $Tb$ theorem for singular integrals had ``local $Tb$ type'' testing conditions for the operator $T$ on pseudo-accretive collections $\{b_Q^i\}$ for $i=1,...,m$, but tested the adjoints of $T$ on the constant function $1$; see \cite{GHO} for more details on this. 
%
In this note, we make a few observations to show that this result can be extended to the following local $Tb$ theorem for multilinear singular integral operators.

\begin{theorem}\label{t(1,1)}
Let $T$ be a continuous $m$ linear operator from $\S\times\cdots\times\S$ into $\S'$ with a standard Calder\'on-Zygmund kernel $K$.  Suppose that $T\in WBP$ and there exist $2\leq q<\infty$ and $1<q_{i,j}<\infty$ with $\frac{1}{q_j}=\sum_{i=1}^m\frac{1}{q_{i,j}}$ and an $m$-compatible collection of functions $\{b_Q^{i,j}\}$ indexed by dyadic cubes $Q$ and $j=1,...,m$ such that
\begin{align}
&\int_Q\(\int_0^{\ell(Q)}|Q_tT^{*j}(P_tb_Q^{1,j},...,P_tb_Q^{m,j})(x)|^2\frac{dt}{t}\)^\frac{q_j}{2}dx\less|Q|.\label{Tcancel}
\end{align}
Then $T$ is bounded from $L^{p_1}\times\cdots\times L^{p_m}$ into $L^p$ for all $1<p_i<\infty$ such that $\frac{1}{p}=\frac{1}{p_1}+\cdot+\frac{1}{p_m}$.  Here we assume that this estimate holds for any approximation to the identity $P_t$ with smooth compactly supported kernels and any Littlewood-Paley-Stein projection operators $Q_t$ whose kernels also are smooth compactly supported function.
\end{theorem}

The precise meaning of \eqref{Tcancel} is the following:  For any $\varphi,\psi\in C_0^\infty$ such that $\widehat\varphi(0)=1$ and $\widehat\psi(0)=0$, \eqref{Tcancel} holds for $P_tf=\varphi_t*f$, $Q_tf=\psi_t*f$, $\varphi_t(x)=\frac{1}{t^n}\varphi(\frac{x}{t})$, and $\psi_t(x)=\frac{1}{t^n}\psi(\frac{x}{t})$, where the constant is independent of the dyadic cube $Q$, but may depend on $\varphi$ and $\psi$.

We prove this theorem by applying the square function estimates that were also proved in \cite{GHO}, but with a few minor modifications to allow for the extension to Theorem \ref{t(1,1)}.  This note is intended to be an addendum to the article \cite{GHO}.  So the reader should refer to that article for definitions, discussion, and history related to these results.

\section{A Few Definitions and Results from \cite{GHO}}

Define the family of multilinear of operators $\{\Theta_{t}\}_{t>0}$ 
\begin{align}
    \Theta_t(f_1,...,f_m)(x)=\int_{\R^{mn}}\theta_t(x,y_1,...,y_m)\prod_{i=1}^mf_i(y_i)dy_i\label{theta}
\end{align}
where $\theta_t:\R^{(m+1)n}\rightarrow\C$ and the square function
\begin{align}
S(f_1,...,f_m)(x)=\(\int_0^\infty|\Theta_t(f_1,...,f_m)(x)|^2\frac{dt}{t}\)^\frac{1}{2}\label{sqfunction}
\end{align}
associated to $\{\Theta_t\}_{t>0}$, where   $f_i$ for $i=1,...,m$ are initially functions in $ C_{0}^{\infty}(\R^{n})$.  Also assume that $\theta_t$ satisfies for all $x,y_1,...,y_m,x', y_1',...,y_m'\in\R^n$
\begin{align}
&|\theta_t(x,y_1,...,y_m)|\less\frac{t^{-mn}}{\prod_{i=1}^m(1+t^{-1}|x-y_i|)^{N+\gamma}}\label{size}\\
&|\theta_t(x,y_1,...,y_m)-\theta_t(x,y_1,...,y_i',...,y_m)|\less t^{-mn}(t^{-1}|y_i-y_i'|)^\gamma\label{regy}\\
&|\theta_t(x,y_1,...,y_m)-\theta_t(x',y_1,...,y_m)|\less t^{-mn}(t^{-1}|x-x'|)^\gamma\label{regx}
\end{align}
for some $N>n$ and $0<\gamma\leq1$.  The following results was proved in \cite{GHO}.
\begin{oldtheorem}
Let $\Theta_t$ and $S$ be defined as in \eqref{theta} and \eqref{sqfunction} where $\theta_t$ satisfies \eqref{size}-\eqref{regx}.  Suppose there exist $q_i,q>1$ for $i=1,...,m$ with $\frac{1}{q}=\sum_{i=1}^m\frac{1}{q_i}$ and functions $b_Q^i$ indexed by dyadic cubes $Q\subset\R^n$ for $i=1,...,m$ such that for every dyadic cube $Q$
\begin{align}
&\int_{\R^n}|b_Q^i|^{q_i}\leq B_1|Q|\label{bsize}\\
&\frac{1}{B_2}\leq\left|\frac{1}{|Q|}\int_Q\prod_{i=1}^mb_Q^i(x)dx\right|\label{baccretive}\\
&\left|\frac{1}{|R|}\int_R\prod_{i=1}^mb_Q^i(x)dx\right|\leq B_3\prod_{i=1}^m\left|\frac{1}{|R|}\int_Rb_Q^i(x)dx\right|\label{bcompatible}\\
&\hspace{4cm}\text{ for all dyadic subcubes }R\subset Q\notag\\
&\int_Q\(\int_0^{\ell(Q)}|\Theta_t(b_Q^1,...,b_Q^m)(x)|^2\frac{dt}{t}\)^\frac{q}{2}dx\leq B_4|Q|.\label{thetacancel}
\end{align}
Then $S$ satisfies 
\begin{align}
||S(f_1,...,f_m)||_{L^p}\less\prod_{i=1}^m||f||_{L^{p_i}}\label{Lpbound}
\end{align}
for all $1<p_i<\infty$ and $2\leq p<\infty$ satisfying $\frac{1}{p}=\frac{1}{p_1}+\cdots+\frac{1}{p_m}$.
\end{oldtheorem}

If $\{b_Q\}$ satisfies \eqref{bsize} and \eqref{baccretive}, we say that $\{b_Q\}$ is a pseudo-accretive system.  We say that $\{b_Q^i\}$ for $i=1,...,m$ is an $m$-compatible, or just compatible, collection of pseudo-accretive systems if they satisfy \eqref{bsize}-\eqref{bcompatible}.  


\section{Proof of Theorem \ref{t(1,1)}}


\begin{proof}[Proof of Theorem \ref{t(1,1)}]
Let $P_t$ be a smooth approximation to identity operators with smooth compactly supported kernels that $P_t^2f\rightarrow f$ as $t\rightarrow0^+$ and $P_t^2f\rightarrow0$ as $t\rightarrow\infty$ in $\S$ for $f\in \S_0$.  Here $\S_0$ is the subspace of Schwartz functions satisfying $|\widehat f(\xi)|\leq C_M|\xi|^M$ for all $M\in\N$.  There exist Littlewood-Paley-Stein projection operators $Q_t^{(i)}$ for $i=1,2$ with smooth compactly supported kernels such that $t\frac{d}{dt}P_t^2=Q_t^{(2)}Q_t^{(1)}=Q_t$.  Using these operators, we decompose $T$ for $f_i\in\S_0$, $i=0,...,m$
\begin{align}
\<T(f_1,...,f_m),f_0\>&=\int_0^\infty t\frac{d}{dt}\<T(P_t^2f_1,...,P_t^2f_m),P_t^2f_0\>\frac{dt}{t}\notag\\
&=\sum_{i=0}^m\<\int_0^\infty Q_tT(P_t^2f_1,...,P_tf_{i-1},P_tf_0,P_tf_{i+1}...,P_t^2f_m)\frac{dt}{t},f_i\>\notag\\
&=\sum_{i=0}^m\<T_i(f_1,...,f_{i-1},f_0,f_{i+1},...,f_m),f_i\>,\label{Tdecomp}
\end{align}
where we take the last line in \eqref{Tdecomp} as the definition of $T_i$ for $i=0,1,...,m$.  It follows that $T_i$ is a multilinear singular integral operator with standard kernel
\begin{align*}
K_i(x,y_1,...,y_m)=\int_0^\infty\<T^{*i}(\varphi_t^{y_1},...,\varphi_t^{y_m}),\psi_t^x\>\frac{dt}{t},
\end{align*}
Also let $\Theta_t^i$ be the multilinear operator associated to 
\begin{align*}
\theta_t^i(x,y_1,...,y_m)=Q_t^{(1)}T^{*i}(\varphi_t^{y_1},...,\varphi_t^{y_m})(x),
\end{align*}
and let $S_i$ be the square function associated to $\Theta_t^i$.  Note that $\theta_t^i(x,y_1,...,y_m)\neq \<T^{*i}(\varphi_t^{y_1},...,\varphi_t^{y_m}),\psi_t^x\>$, and that $T_i$ is not actually the integral of $\Theta_t^i$ (one has $Q_t$ and the other has $Q_t^{(1)}$).  Furthermore, by the hypotheses on $T^{*i}$ and by the local $Tb$ theorem for square functions from \cite{GHO}, it follows that $S_i$ is bounded from $L^{2m}\times\cdots\times  L^{2m}$ into $L^2$.  Therefore we have
\begin{align*}
\<T_i(f_1,...,f_m),f_0\>&=\int_0^\infty \int_{\R^n}Q_tT(P_t^2f_1,...,P_t^2f_m)(x)f_0(x)dx\frac{dt}{t}\\
&=\int_0^\infty \int_{\R^n}Q_t^{(1)}T(P_t^2f_1,...,P_t^2f_m)(x)Q_t^{(2)\,*}f_0(x)dx\frac{dt}{t}\\
&\leq||S_i(f_1,...,f_m)||_{L^2}\left|\left|\(\int_0^\infty|Q_t^{(2)\,*}f_0|^2\frac{dt}{t}\)^\frac{1}{2}\right|\right|_{L^2}\less||f_1||_{L^{2m}}\cdots||f_m||_{L^{2m}}||f_0||_{L^2}.
\end{align*}
Hence $T_i$ is bounded from $L^{2m}\times\cdots\times L^{2m}$ into $L^2$, and by the multilinear Calder\'on-Zygmund theory developed by Grafakos and Torres in \cite{GT1,GT2}, it follows that $T_i$ also bounded from $L^{p_1}\times\cdots\times L^{p_m}$ into $L^p$ for all $1<p_1,...,p_m<\infty$ with $\frac{1}{p}=\frac{1}{p_1}+\cdots+\frac{1}{p_m}$.  In particular, it follows that $T_i$ is bounded from $L^{m+1}\times\cdots\times L^{m+1}$ into $L^\frac{m+1}{m}$ for each $i=0,1,...,m$.  Then continuing from \eqref{Tdecomp}, we have
\begin{align*}
|\<T(f_1,...,f_m),f_0\>|&\leq\sum_{i=0}^m|\<T_i(f_1,...,f_{i-1},f_0,f_{i+1},...,f_m),f_i\>|\\
&\leq\sum_{i=0}^m||T_i(f_1,...,f_{i-1},f_0,f_{i+1},...,f_m)||_{L^\frac{m+1}{m}}||f_i||_{L^{m+1}}\less\prod_{j=0}^m||f_j||_{L^{m+1}}.
\end{align*}
Therefore $T$ is bounded from $L^{m+1}\times\cdots L^{m+1}$ into $L^\frac{m+1}{m}$.  Then it again follows from the multilinear Calder\'on-Zygmund theory in \cite{GT1,GT2} that $T$ is bounded from $L^{p_1}\times\cdots\times L^{p_m}$ into $L^p$ for all $1<p_1,...,p_m<\infty$ such that $\frac{1}{p}=\frac{1}{p_1}+\cdots+\frac{1}{p_m}$.
\end{proof}


\bibliographystyle{amsplain}

\end{document}